%% file: MainQM-II.tex
\documentclass[12pt,reqno]{amsart}

 \newif\ifdraft
 \drafttrue            

\ifdraft
\usepackage[article]{MetaPackage}
\else
\usepackage{MetaPackage}
\fi

\newcommand{\Cexp}{C_{\mrm{exp}}}

\DeclareFunSpace{QCB}{\mathrm{QCB}}
\DeclareFunSpace{QCBt}{\lift{\mathrm{QCB}}}

\title[Quasimorphisms on Higher Rank Lattices]{Thermodynamic Formalism for Quasimorphisms: Lattices in Higher Rank Semisimple Lie Groups}

\author{Pablo D. Carrasco}
\address{ICEx-UFMG, Avda. Presidente Antonio Carlos 6627, Belo Horizonte-MG, BR31270-90}
\email{pdcarrasco@gmail.br}

\author{Federico Rodriguez-Hertz}
\address{Penn State, 227 McAllister Building, University Park, State College, PA16802}
\email{hertz@math.psu.edu}

 \date{\today}

\begin{document}

\begin{abstract}
We give a new proof, based on thermodynamic formalism, of a foundational result of Burger and Monod in bounded cohomology. Let $G$ be a noncompact connected semisimple real Lie group with finite center and no factors of real rank one, and let $\Gamma<G$ be a uniform lattice. We prove that, for every orthogonal representation $\pi:\Gamma\to\On_N$, every $\pi$-quasimorphism $L:\Gamma\to\R^N$ is bounded.
\end{abstract}

\maketitle

\input{1IntroductionNEW}

\input{2QuasiCocycleQuasiMorphism}

\input{3EqStCenIso}

\input{4ProofthmNEW}


\printbibliography[title={References}]
\end{document}

%% file: 1IntroductionNEW.tex

\section{Introduction}\label{sec:introduction}

Bounded cohomology provides a natural framework for the following natural rigidity question. Given a map satisfying the cocycle identity up to a uniformly bounded error, must it lie at bounded distance from a genuine cocycle? Let $\Gamma$ be a discrete group and let $\pi:\Gamma\to\On_N$ be an orthogonal representation. The inclusion of the complex of bounded cochains into the ordinary cochain complex induces comparison maps
\[
    c_n:H^n_{\mrm b,\pi}(\Gamma,\R^N)\longrightarrow H^n_{\pi}(\Gamma,\R^N);
\]
$\Ker c_n$ measures the additional information carried by $H^n_{\mrm b,\pi}$ compared to its algebraic counterpart $H^n_{\pi}$. The first potentially nontrivial instance occurs in degree two; this is the most interesting and studied case, and it is the one that we are focusing on here.

The kernel of $c_2$ can be described in terms of approximate $1$-cocycles. A map $L:\Gamma\to\R^N$ is called a $\pi$-quasimorphism if
\[
    \norm{\delta_\pi L} \defeq  \sup_{\gamma,\eta\in\Gamma} |L(\gamma\eta)-L(\gamma)-\pi(\gamma)L(\eta)|<\oo.
\]
The bounded cocycle $\delta_\pi L$ determines a class in $\Ker c_2$, and every class in $\Ker c_2$ arises in this way. It follows that $c_2$ is injective if and only if every $\pi$-quasimorphism differs by a bounded map from a genuine $\pi$-cocycle.

The behavior of this kernel sharply distinguishes familiar classes of groups. If $\Gamma$ is amenable, then $H^2_{\mrm b,\pi}(\Gamma,\R^N)=0$, and thus $\Ker c_2=0$.\footnote{More generally, the bounded cohomology of an amenable group vanishes in every positive degree with coefficients in any dual Banach module \cite{JohnsonBanach}.} By contrast, if $\Gamma$ is a non-elementary Gromov-hyperbolic group, then $\Ker c_2$ is infinite-dimensional for the trivial representation. This phenomenon originates in the counting quasimorphisms introduced by Brooks and in the work of Brooks--Series on surface groups, and was subsequently developed in much greater generality \cite{Brooksbounded,BrooksSeries84,Barge1988,Epstein1997,Fujiwara1998,Bestvina2002,Hamenstaedt2008,Bestvina_2016}. In recent work, we used thermodynamic formalism to completely describe this kernel and obtain precise statistical information on quasimorphisms of negatively curved groups \cite{TermoEqSt}.

The picture is rigid for higher-rank lattices. Burger and Monod proved that the comparison map in degree two is injective for irreducible lattices in higher-rank semisimple groups, including nontrivial unitary coefficients, subject to the appropriate nondegeneracy condition when rank-one simple factors are present \cite{Burger1999,Burger2002}. The purpose of the present paper is to give a new proof, based on thermodynamic formalism and homogeneous dynamics, of the following finite-dimensional version.

\begin{definition}
    A lattice $\Gamma<G$ has no factors of real rank one if, for every nontrivial connected normal quotient $q:G\longrightarrow Q$ such that $Q$ has no compact factors and $\rank_{\R}Q=1$, the subgroup $q(\Gamma)$ is dense in $Q$.
\end{definition}

\begin{maintheorem}\label{thm:mainA}
Let $G$ be a noncompact connected semisimple real Lie group with finite center, and let $\Gamma<G$ be a uniform lattice with no factors of real rank one. For every finite dimensional orthogonal representation $\pi:\Gamma\to\On_N$, every $\pi$-quasimorphism $L:\Gamma\to\R^N$ is bounded.
\end{maintheorem}

In particular, the comparison map $c_2$ is injective. The theorem also applies to genuine $\pi$-cocycles; in that case boundedness implies that
\[
    L(\gamma)=v-\pi(\gamma)v
\]
for some $v\in\R^N$. Thus, the vanishing of $H^1_\pi(\Gamma,\R^N)$ is recovered directly from our argument; no independent group-cohomological vanishing result is used in the proof.

Compact factors of $G$ are allowed, and $\Gamma$ need not be irreducible. For example, the theorem applies when
\[
    G=\prod_{i=1}^p G^{(i)},
    \qquad
    \Gamma=\prod_{i=1}^p\Gamma^{(i)},
\]
where each $\Gamma^{(i)}<G^{(i)}$ is an irreducible uniform lattice and
\[
    \rank_{\R}G^{(i)}\geq 2
    \qquad\text{for every }i.
\]

Let us locate \Cref{thm:mainA} more precisely within the existing literature. In the almost-simple higher-rank case, Burger--Monod prove injectivity of the comparison map for arbitrary unitary coefficients and without the cocompactness assumption \cite{Burger1999,Burger2002}. A central ingredient in their proof is $L^2$-induction, which embeds the second bounded cohomology of $\Gamma$ into the continuous bounded cohomology of the ambient group $G$, with coefficients in the induced Hilbert $G$-module $L^2\text{-}\operatorname{Ind}_\Gamma^G(\pi)$. This module is modeled on $\Lp{2}{\Gamma\backslash G,\mathcal H}$, where $\mathcal H$ is the coefficient space of $\pi$, and is infinite-dimensional even when $\mathcal H$ is finite-dimensional. Their result implies that every $\pi$-quasimorphism is the sum of a bounded map and a genuine $\pi$-cocycle. To conclude that the quasimorphism itself is bounded, one must additionally invoke the vanishing of $H^1_\pi$, which for unitary coefficients follows from property~$(T)$.

Similarly, Monod--Shalom identify degree-two bounded cohomology with the part carried by invariant vectors and prove that every rough affine action on a separable reflexive Banach space is a bounded perturbation of an actual action \cite[Theorem~6.3 and Corollary~6.6]{Monod2004}. Obtaining a bounded orbit for the resulting action again requires an independent first-cohomology rigidity result. Thus, together with the corresponding vanishing of $H^1_\pi$, these results contain the almost-simple case of \Cref{thm:mainA}. By contrast, our argument proves boundedness of the quasimorphism directly and consequently yields, within the finite-dimensional orthogonal coefficient class considered here, both the injectivity of $c_2$ and the vanishing of $H^1_\pi$, without invoking property~$(T)$ or any independent group-cohomological vanishing theorem.

The passage from irreducible blocks to the general statement also deserves some clarification. At the level of bounded cohomology, products are governed by the degree-two product formula of Burger--Monod and Monod, stated in \cite[Theorem~3.7]{Monod_2006}. This formula is not a formal Künneth theorem and depends essentially on the hypotheses imposed on the coefficient module. Our proof does not use this product formula. After passing to a finite-index subgroup that decomposes as a product of irreducible blocks, we restrict the quasimorphism to each block. The irreducible case makes each restriction bounded, and the quasicocycle identity then implies boundedness on the entire product. Boundedness subsequently passes from the finite-index subgroup to the original lattice. Thus, although the reducible conclusion can also be recovered from the existing bounded-cohomological theory, its reduction to the irreducible case is direct in our approach.

The works of Burger--Monod and Monod--Shalom discussed above are foundational in the bounded-cohomological theory of higher-rank lattices, and the conclusion of \Cref{thm:mainA} represents only one consequence of their much broader frameworks. The novelty of the present work lies in the method and in the dynamical rigidity principle underlying the proof. Our guiding principle is that quasimorphisms, considered up to bounded error, can be encoded by invariant probability measures arising from natural dynamical systems associated with the group or with a related homogeneous space. More precisely, each quasimorphism gives rise to a canonical invariant measure, and two quasimorphisms with the same canonical measure and the same mean differ by a uniformly bounded map. Thus, the measure together with the mean determines the quasimorphism up to bounded error. In the rank-one setting studied in \cite{TermoEqSt}, the underlying dynamics is hyperbolic and supports a rich family of invariant measures. By contrast, in the higher-rank setting considered here, homogeneous measure rigidity severely restricts the measures that can arise and ultimately forces the associated measure to be Haar, leading to the boundedness conclusion.

In this article we develop a version of thermodynamic formalism for a class of quasicocycles over center isometries, which we call Bowen quasicocycles, and we show that every such quasicocycle determines a canonical Gibbs-type invariant measure. The quasimorphisms appearing in \Cref{thm:mainA} give rise to Bowen quasicocycles on a compact homogeneous extension of \(\Gamma\backslash G\). Thermodynamic uniqueness, combined with homogeneous measure rigidity, forces the corresponding canonical measures to be Haar and ultimately yields the boundedness of the original quasimorphisms.

\subsection{Outline of the proof}

Let $C=\overline{\pi(\Gamma)}<\On_N$ and consider
\[
    G'=G\times C,
    \qquad
    \Gamma'=\{(\gamma,\pi(\gamma)): \gamma\in\Gamma\},
    \qquad
    M'=\Gamma'\backslash G'.
\]
Since $\Gamma$ is uniform and $C$ is compact, $M'$ is compact. Starting from a $\pi$-quasimorphism $L$, we construct a quasicocycle $\ms S$ for the right action of $G'$ on $M'$. This is the coarse analogue of inducing a cocycle from the lattice to the ambient group, but the construction remains finite-dimensional and geometric.

We next restrict $\ms S$ to suitable semisimple elements $a\in G$. The elements we use lie on walls of Weyl chambers. Their right translations on $M'$ are center isometries: they are isometric along a center foliation and hyperbolic in the transverse directions. For every linear functional $\ell\in(\R^N)^*$, the scalar sequence obtained from $\ell\circ\ms S$ along the iterates of $a$ has well defined asymptotic control (a weak version of the so called Bowen property). To study this sequence, we develop a theory of thermodynamic formalism for center isometries which in particular, associates
to each such quasicocycle a distinguished invariant probability measure satisfying a Gibbs-type estimate.

The choice of $a$ on a Weyl chamber wall provides large noncompact subgroups in its centralizer. The Gibbs property, together with the higher-rank hypothesis, makes the distinguished measure invariant and ergodic under sufficiently many of these subgroups. A version of Ratner's measure classification theorem then forces it to be Haar measure on $M'$. Comparing each scalar restriction of $\ms S$ with the constant
additive cocycle given by its Haar mean, thermodynamic rigidity shows that their difference is uniformly bounded.

It remains to pass from the selected wall directions to all of $G$. The Cartan decomposition reduces this to an additive datum on a maximal split torus $A$. The preceding argument makes this datum invariant under the Weyl group. The Weyl reflections then force it to vanish, and the compact factors in the Cartan decomposition contribute only a uniform error. We conclude that $\ms S$ is uniformly bounded and, hence, that $L$ is bounded.

The same thermodynamic construction has a natural analogue for semisimple $p$-adic groups and their Bruhat--Tits buildings. We do not pursue that extension here.

\subsection{A higher-degree question}

The contrast between the hyperbolic and higher-rank cases is reminiscent of the philosophy introduced by Katok and Spatzier in their study of the first cohomology of higher-rank abelian hyperbolic actions \cite{Katok_1994}. For the higher-rank Anosov actions considered there, smooth cocycles can be represented by leafwise closed one-forms, and Liv\v{s}ic theory together with commutativity forces the first cohomology to reduce to the finite-dimensional space of constant cocycles. Katok and Katok later established related higher-degree results for suitable $\mathbb Z^r$-actions \cite{Katok_1995}. It is natural to ask whether some part of this dynamical philosophy extends to bounded cohomology in higher degree.

Any such question must take into account a restriction theorem of Monod. If $G$ is almost simple, $r=\rank_{\R}G$, and $\Gamma<G$ is a irreducible lattice, then the restriction map
\[
H^j_{\mrm{b}}(G,\R)
\longrightarrow
H^j_{\mrm b}(\Gamma,\R)
\]
is an isomorphism for $j<2r$ \cite[Corollary~4.8]{Monod2010}. Thus, throughout the range considered below, the question for the lattice is equivalent to the corresponding question for the ambient group. Moreover, Monod's Isomorphism Conjecture predicts that the comparison map
\[
H^j_{\mrm{cb}}(G,\R) \longrightarrow H^j_{\mrm c}(G,\R)
\]
is an isomorphism in every degree \cite[Problem A]{MonodICM_2007}. Since continuous cohomology is finite-dimensional, the following question is therefore a deliberately weaker form of that conjecture.

\begin{question}
Let $G$ be a noncompact connected almost-simple real Lie group with finite center, let $\Gamma<G$ be a lattice, and set $r=\rank_{\R}G$. Is
\[
\dim H^j_{\mrm b}(\Gamma,\R)<\oo
\]
for every $3\leq j\leq r$? Equivalently, is $H^j_{\mrm{cb}}(G,\R)$ finite-dimensional in this range?
\end{question}

A stronger question, consistent with the Isomorphism Conjecture, is whether
\[
H^j_{\mrm{cb}}(G,\R)={0}
\]
in those degrees $3\leq j\leq r$ for which $H^j_{\mrm c}(G,\R)={0}$. The qualification involving continuous cohomology
is essential. Groups of Hermitian type provide an evident obstruction through their bounded K\"ahler classes and cup products, but the absence of Hermitian type alone does not rule out nontrivial continuous cohomology in higher degree.

For $G=\SL_d(\R)$, $d\geq3$, the known low-degree results give
\[
H^j_{\mrm b}(\Gamma,\R)={0}
\qquad (j=1,2,3)
\]
for every lattice $\Gamma<\SL_d(\R)$. Indeed, degree one vanishes for every group, degree two follows from Burger--Monod, and degree three follows from
Monod's vanishing of $H^3_{\mrm{cb}}(\SL_d(\R),\R)$ together with the restriction theorem \cite{Burger2002,Monod2004,Monod2010}. The remaining degrees in the range of the question are not covered by these results and appear to be open.

%% file: 2QuasiCocycleQuasiMorphism.tex

\section{Quasicocycles and twisted quasimorphisms}\label{sec:quasicocyclequasimor}

We work in the hypotheses of \Cref{thm:mainA}: $G$ is a noncompact connected semisimple real Lie group with finite center, and let $\Gamma<G$ is a uniform lattice with no factors of real rank one. 

We first make some algebraic reductions in $G$, which will allow us to assume that $G$ has no compact factors, and  is adjoint. After that, we introduce a bundle extension of $\Gamma\backslash G$ that will be useful to control the action of $\pi$.

\subsection{Irreducible factors and reductions of the ambient group}

We first make two reductions concerning the global form of $G$. Let $G_{\mathrm c}\triangleleft G$ be the maximal connected compact normal
subgroup and set
\[
G_{\mathrm{nc}}=G/G_{\mathrm c}.
\]
Denote the quotient homomorphism by $q_{\mathrm{nc}}:G\longrightarrow G_{\mathrm{nc}}$. Then $G_{\mathrm{nc}}$ is semisimple with finite center and has no nontrivial connected compact normal subgroups (no compact almost-simple factors). Consider its adjoint quotient
\[
G_{\mathrm{ad}}
=G_{\mathrm{nc}}/Z(G_{\mathrm{nc}}),
\]
with quotient homomorphism $r:G_{\mathrm{nc}}\longrightarrow G_{\mathrm{ad}}$. Set
\[
q=r\circ q_{\mathrm{nc}}:G\longrightarrow G_{\mathrm{ad}}.
\]
Its kernel fits into the exact sequence
\[
1\longrightarrow G_{\mathrm c}
\longrightarrow\ker q
\longrightarrow Z(G_{\mathrm{nc}})
\longrightarrow1
\]
and in particular is compact. Therefore, $F:=\Gamma\cap\ker q$ is a finite normal subgroup of $\Gamma$, and
\[
\overline{\Gamma}:=q(\Gamma)<G_{\mathrm{ad}}
\]
is a uniform lattice.

Let $L:\Gamma\to\R^N$ be a $\pi$-quasimorphism. We claim that the finite kernel $F$ can be removed without affecting the
boundedness problem. Denote $V=\mathbb R^N$ and let
\[
P_F
=
\frac{1}{|F|}
\sum_{f\in F}\pi(f):V\longrightarrow V^F
\]
be the orthogonal projection onto the subspace of $F$-invariant vectors. Denote $C_F=\max_{f\in F}\abs{L(f)}$.

For $f\in F$ and $\gamma\in\Gamma$, comparison of the two decompositions
\[
f\gamma=\gamma(\gamma^{-1}f\gamma)
\]
and the fact that $\gamma^{-1}f\gamma\in F$ give
\[
\abs{(\pi(f)-\Id)L(\gamma)}\leq 2\norm{\delta_\pi L}+2C_F.
\]
Averaging over $f\in F$, we obtain
\[
\sup_{\gamma\in\Gamma}\abs{L(\gamma)-P_FL(\gamma)}<\oo.
\]

Since $F\triangleleft\Gamma$, the subspace $V^F$ is $\Gamma$-invariant and $P_F$ commutes with $\pi(\Gamma)$. The restriction of $\pi$ to $V^F$
therefore factors through an orthogonal representation 
\[
\overline{\pi}:\overline{\Gamma}\longrightarrow O(V^F).
\]
Choose a section
\[
s:\overline{\Gamma}\longrightarrow\Gamma
\]
of $q|_\Gamma$ and define
\[
\overline L(\overline\gamma)
=
P_FL(s(\overline\gamma)).
\]
Then $\overline L$ is an $\overline\pi$-quasimorphism. Indeed, if $\overline\gamma_1,\overline\gamma_2\in\overline\Gamma$, then
\[
s(\overline\gamma_1)s(\overline\gamma_2)=f\,s(\overline\gamma_1\overline\gamma_2)
\]
for some $f\in F$, and 
\[
\abs{\overline L(\overline\gamma_1\overline\gamma_2)-\overline L(\overline\gamma_1)
-\overline\pi(\overline\gamma_1)\overline L(\overline\gamma_2)
}
\leq
2\norm{\delta_{\pi}L}+C_F.
\]
Moreover, regarding $V^F$ as a subspace of $V$, we have
\[
\sup_{\gamma\in\Gamma}\abs{L(\gamma)-\overline L(q(\gamma))}<\infty.
\]
Thus it suffices to prove that $\overline L$ is bounded.


Relabelling the objects obtained above, we may assume without loss of generality that $G$ is adjoint. In particular, $G$ has no compact factors,
has trivial center, and is algebraic.

We next pass to an algebraically simply connected global form of $G$. Recall that a connected semisimple real algebraic group $\mathbf H$ is
said to be algebraically simply connected if every surjective algebraic homomorphism
\[
\mathbf H'\longrightarrow\mathbf H
\]
from a connected semisimple real algebraic group $\mathbf H'$, whose kernel is finite and contained in the center of $\mathbf H'$, is an
isomorphism.

Let $\mathbf G$ be the connected adjoint real algebraic group such that $G=\mathbf G(\mathbb R)^\circ$. Every connected semisimple algebraic group admits an algebraically simply connected covering
\[
\mathbf p:\widetilde{\mathbf G}\longrightarrow\mathbf G.
\]
which is unique up to isomorphism, where the homomorphism $\mathbf p$ is surjective, and whose differential at the identity is an isomorphism. Therefore $\ker\mathbf p$ is finite, and since it is a finite normal subgroup in the connected group $\widetilde{\mathbf G}$, it is contained in the center of $\widetilde{\mathbf G}$. See \cite{Milne_2017}, Chapter $18$.

Let $\widetilde G=\widetilde{\mathbf G}(\mathbb R)^\circ$. The restriction of $\mathbf p$ gives a surjective covering homomorphism
\[
p:\widetilde G\longrightarrow G
\]
with finite central kernel $Z = \ker p$. Indeed, its differential is an isomorphism, so its image is an open subgroup of the connected group $G$, and hence equals $G$.

Consider $\widehat\Gamma:=p^{-1}(\Gamma)<\widetilde G$. The covering map $p$ induces a homeomorphism
\[
\widehat\Gamma\backslash\widetilde G
\longrightarrow
\Gamma\backslash G,
\qquad
\widehat\Gamma\widetilde g
\longmapsto
\Gamma p(\widetilde g).
\]
and in particular $\widehat\Gamma$ is a uniform lattice in $\widetilde G$.

Notice that $\widehat\Gamma$ need not be torsion-free, even if $\Gamma$ is torsion-free, since it contains the finite group $Z$. Since $\widetilde G$
is linear and $\widehat\Gamma$ is finitely generated, Selberg's lemma provides a torsion-free finite-index subgroup
\[
\widetilde\Gamma<\widehat\Gamma.
\]
After replacing it by its normal core, we may assume that $\widetilde\Gamma$ is normal in $\widehat\Gamma$.

Since $Z$ is finite and $\widetilde\Gamma$ is torsion-free, $\widetilde\Gamma\cap Z=\{e\}$, and therefore the restriction of $p$ to $\widetilde\Gamma$ is injective. Setting
\[
\Gamma_0:=p(\widetilde\Gamma),
\]
we obtain a finite-index subgroup $\Gamma_0<\Gamma$ such that
\[
p|_{\widetilde\Gamma}:
\widetilde\Gamma\longrightarrow\Gamma_0
\]
is an isomorphism. Thus, after passing to a finite-index subgroup, the lattice $\Gamma$ lifts to the algebraically simply connected cover.

We transport the representation and the quasimorphism to $\widetilde\Gamma$ by setting
\[
\widetilde\pi(\widetilde\gamma)
=\pi\bigl(p(\widetilde\gamma)\bigr),
\qquad
\widetilde L(\widetilde\gamma)
=L\bigl(p(\widetilde\gamma)\bigr).
\]
Boundedness of $\widetilde L$ is equivalent to boundedness of $L|_{\Gamma_0}$. Since $\Gamma_0$ has finite index in $\Gamma$, boundedness
on $\Gamma_0$ implies boundedness on $\Gamma$.

Relabelling
\[
(\widetilde G,\widetilde\Gamma,\widetilde\pi,\widetilde L)
\]
as
\[
(G,\Gamma,\pi,L),
\]
we may therefore assume from now on that $G$ is algebraically simply connected, has no compact factors, and that $\Gamma<G$ is a torsion-free
uniform lattice.

The condition of having no factors of real rank one is preserved under the finite-index passages and finite central coverings used above.
Decompose the almost-simple factors of $G$ into maximal irreducible blocks
\[
G=G^{(1)}\times\cdots\times G^{(p)}
\]
with respect to $\Gamma$. Then $\Gamma$ has no factors of real rank one
if and only if
\[
\rank_{\mathbb R}G^{(i)}\geq2
\qquad
\text{for every }i=1,\ldots,p.
\]
Indeed, a block of real rank one gives a rank-one normal quotient on which the image of $\Gamma$ is discrete, whereas irreducibility implies
that the projection to every proper factor of a block is dense.

After passing to a finite-index subgroup and relabelling, we may moreover
assume that
\[
\Gamma=\Gamma^{(1)}\times\cdots\times\Gamma^{(p)},
\]
where each $\Gamma^{(i)}<G^{(i)}$ is an irreducible torsion-free uniform lattice.

\subsection{The bundle extension}

Let $M=\Gamma\backslash G$, and consider the natural right action $G\acts M$. For the representation $\pi:\Gamma\to \On_N$ we denote $\mc C=\cl{\Im\pi}$ and let $G'=G\times \mc C$. The group $\Gamma$ acts on the left on $G'$ by $\gamma\cdot(g,u)=(\gamma g,\pi(\gamma)u)$; alternatively, we consider 
\[
	\Gamma':=\{(\gamma,\pi(\gamma)):\ga\in\Gamma\}<G'
\]
with the usual left action, and denote $M'=\Gamma'\backslash G'$. Note that $p_M:M'\to M, p(\overline{(g,h)})=\overline{g}$ makes $M'$ a compact flat $\mc C$-bundle, and in particular $\Gamma'$ is a uniform lattice in $G'$, which does not have rank-one factors. Note also that the action $G\acts M$ extends naturally to an action $G'\acts M'$. 

It will be convenient to use the following equivalent description of the above action. Let $\Sigma \subset G$ be a fundamental domain for $M$, and $\Sigma' \subset G'$ a fundamental domain for $M'$. With no loss of generality, we assume that $\Sigma'=\Sigma\times\mc C$.

Define the integer part maps $[\cdot]:G\to \Gamma, [\cdot]:G'\to \Gamma'$ be such that
\begin{align*}
&[g]^{-1}g\in \Sigma\\
&[(g,u)]^{-1}(g,u)\in \Sigma'.
\end{align*}
Note that $[g,u]=([g],\pi([g]))$, and the abuse of notation of calling both maps the same is harmless.  The following is clear.

\begin{lemma}
	The functions $[\cdot]:G\to \Gamma, [\cdot]:G'\to \Gamma'$ are $\Gamma$- and $\Gamma'$-equivariant: for $\gamma\in \Gamma, g\in G, u\in\mc C$
	\begin{align*}
	&[\gamma g]=\gamma[g]\\
	&[(\gamma,\pi(\gamma))\cdot(g,u)]=(\gamma,\pi(\gamma))[g,u].
	\end{align*}
\end{lemma}

Using the map $[\cdot]$, we define a right action of $G'$ on $\Sigma'$ as follows. For $x'\in \Sigma'$ and $g'\in G'$, set
\[
x'*g':=[x'g']^{-1}x'g'.
\]

\begin{lemma}
	The above formula defines an action of $G'$ on $\Sigma'$. 
\end{lemma}

\begin{proof}
	It is enough to show that for $x\in \Sigma, g\in G$, $x*g:=[xg]^{-1}xg$ defines a right action of $G$ on $\Sigma$. Note that $xg=[xg]x*g$, thus
\begin{align*}
	[x(g_1g_2)]x*(g_1g_2)&=x(g_1g_2)=(xg_1)g_2=([xg_1]x*g_1)g_2\\
	  &=[xg_1]((x*g_1)g_2)\\
    &=[xg_1][(x*g_1)g_2]\left((x*g_1)*g_2\right)\\
    &=[xg_1][[xg_1]^{-1}(xg_1)g_2]\left((x*g_1)*g_2\right)\\
    &=[xg_1g_2]\left((x*g_1)*g_2\right)\\
 	\end{align*}	
and $x*(g_1g_2)=(x*g_1)*g_2$.
\end{proof}

It follows in particular that $\Sigma'\times G'$ is a $G'-$principal bundle; note also that $G'$ acts (on the left) on $\R^N$, the action corresponding to the representation $\pi$. Now given $L:\Gamma\to \R^{N}$ a $\pi$-quasimorphism, we will show that it defines naturally a function on the vector bundle associated to $\pi$. 

Define a map $S:\Sigma'\times G'\to\R^N$ by 
\begin{align}\label{eq:associatedquasicocycle}
S((x,A),(g,B))=A^{-1}L([x g]).
\end{align}
We will show that $S$ is a quasicocycle, as per the next definition.

\begin{definition}\label{def:quasicocycle}
Let $G$ be a group acting on a space $E$, and let $\rho:G\to \GL_n(\R)$ be a representation. A map $S:E\times G\to \R^n$ is called a $\rho$-quasicocycle if
\[
  \sup_{g_i\in G, p\in E}|S(p,g_1g_2)-\rho(g_1)S(pg_1,g_2)-S(p,g_1)|<\oo.
\]

If $\rho$ is the trivial representation, we simply say that $S$ is a quasicocycle for the action.
\end{definition}

For $a,b\in\Gamma$ we write
\[
\delta_{\pi}(L)(a,b)=L(ab)-(\pi(a)L(b)+L(a))	
\]
and let 
\[
C(a):=\pi(a)L(a^{-1})+L(a)=L(e)-\delta_{\pi}L(a,a^{-1});	
\]
notice that $C(a)$ it is uniformly bounded. We now make the explicit computations: for $(x,A)\in \Sigma'$ and $(g_i,B_i)\in G'$, $i=1,2$, we have
\begin{align*}
S((x,A),(g_1g_2,B_1B_2))=A^{-1}L([xg_1g_2])\\
S((x,A),(g_1,B_1))=A^{-1}L([xg_1])
\end{align*}
and
\begin{align*}
&B_1 S\left((x,A)*(g_1,B_1), (g_2,B_2)\right)=B_1S\left([xg_1,AB_1]^{-1}(xg_1,AB_1),(g_2,B_2)\right)\\
&=B_1S\left([xg_1]^{-1}xg_1,\pi([xg_1])^{-1}AB_1,(g_2,B_2)\right)=A^{-1}\pi([xg_1])L\left([xg_1]^{-1}[xg_1g_2]\right)\\
&=A^{-1}\pi([xg_1])\left(L\left([xg_1]^{-1}[xg_1g_2]\right)-\pi([xg_1])^{-1}L([xg_1g_2])-L([xg_1]^{-1})\right)\\
&+A^{-1}\pi([xg_1])\left(\pi([xg_1])^{-1}L([xg_1g_2])+L([xg_1]^{-1})\right)\\
&=\underbrace{A^{-1}\pi([xg_1])\delta_{\pi}L([xg_1]^{-1},[xg_1g_2])}_{\text{bounded}:=K_1}+S((x,A),(g_1g_2,B_1B_2))\\
&+A^{-1}\underbrace{\pi([xg_1])L([xg_1]^{-1})}_{C([xg_1])-L([xg_1])}\\
&=K_1+S((x,A),(g_1g_2,B_1B_2))+\underbrace{A^{-1}C([xg_1])}_{\text{bounded}:=K_2}-S((x,A),(g_1,B_1)).
\end{align*}
To get the bounds $K_1, K_2$ we have used that $\pi$ is orthogonal, hence it preserves norms. Let $K:=K_1+K_2$. We have proven the following.

\begin{proposition}\label{pro:quasicoccycleG}
	There exists a constant $K>0$ such that for all $(x,A)\in\Sigma'$ and $(g_i,B_i)\in G'$, $i=1,2$,
\begin{align*}
\MoveEqLeft\left|S\left((x,A),(g_1g_2,B_1B_2)\right)\right.\\
&\left.-\rho\left((g_1,B_1)\right) S\left((x,A)*(g_1,B_1),(g_2,B_2)\right)-S\left((x,A),(g_1,B_1)\right)\right|\leq K
\end{align*}
where $\rho\left((g,B)\right)=B^{-1}$. That is, $S:\Sigma'\times G'\to\R^N$ is a $\rho-$quasicocycle.

In particular, $S$ is a quasicocycle for the action of $G=G\times\{1\}<G'$ on $\Sigma'\times G'$.
\end{proposition} 

Quasicocycles are natural analogues of cocycles for actions, and play a central role in our arguments. Observe that in particular a quasicocycle for a $\Z-$action on a space $X$  (thus, given by some invertible map $T:X\toit$) is equivalent to a sequence of functions $(S_n:X\to \R^n)$ so that, 
\[
	(\forall n,m), \quad \norml{S_{n+m}-S_n-S_m\circ T^n}\leq C
\]
for some uniform constant $C$. Similarly for semi-actions of $\N$. Thus, in this case we will refer to the sequence $\bm{S}=(S_n)$ as the quasicocycle $\bm S$.

\begin{remark}[Non-uniform lattices.]
The results given in this section also hold for non-uniform lattices. To achieve this one can use that the word and Riemannian metrics on $\Gamma$ are uniformly comparable \cite{Lubotsky2000}. 
\end{remark}




%% file: 3EqStCenIso.tex
\section{Equilibrium states for Center Isometries}\label{sec:existeEq}

We will study the action of certain elements $a\in G$ on $\Sigma'\times G'$. In this section we introduce such elements and develop a thermodynamic formalism adapted to this setting. To avoid unnecessary technicalities, we work in a more general framework.

Consider a laminated space $(X,\ms{W})$, whose leaves are $c$-dimensional smooth manifolds. We assume that $T\ms{W}$ is a continuous bundle over $X$, equipped with a continuous Riemannian metric whose induced distance is locally comparable to that of $X$. For $\delta>0$ and $x\in X$, we denote by $B^{\ms{W}}_\delta(x)$ the local plaque centered at $x$ of radius $\delta$, where distance is measured intrinsically along $\ms{W}(x)$.

Let $T:X\to X$ be a homeomorphism preserving $\ms{W}$.

\begin{definition}\label{def:plaqueexp}
	The pair $(T,\ms{W})$ is called plaque expansive if there exists $\Cexp>0$ so that 
	\[
	(\forall x,y\in X, \forall 0<\del\leq \Cexp):\ \dis[X]{T^nx}{T^ny}\leq \del\Rightarrow y\in B^{\ms{W}}_{3\del}(x).
	\]
\end{definition}

Since $\ms{W}$ will be fixed, we will simply say that \emph{$T$ is plaque expansive}. This notion was introduced by Hirsch, Pugh and Shub \cite{HPS} in the context of partially hyperbolic dynamics, which will be our main example.

\begin{definition}[Center Isometry]
    Suppose that $X$ is a closed smooth manifold. A $\mc{C}^1$ diffeomorphism $T:X\toit$ is a center isometry if there exists a $DT-$invariant splitting $TX=E^s\oplus E^c\oplus E^u$, and a Riemannian metric on $X$ so that
	\begin{enumerate}
	 	\item $DT|E^c$ is an isometry.
	 	\item $\norm{DT|E^s}, \norm{DT^{-1}|E^u}<1$.
	 \end{enumerate} 
\end{definition}

If $T$ is a center isometry we refer to the bundles $\Es, \Eu, \Ec, \Ecs=\Ec\oplus\Es, \Ecu=\Ec\oplus\Eu$ as the \emph{stable}, \emph{unstable}, \emph{center}, \emph{center stable}, and \emph{center unstable}, respectively.

\begin{example}
In the setting of the introduction, there are always elements $a\in G$ so that its induced map $T_a:M'=\Gamma'\backslash G'\toit$ is a center isometry. See the last Section for a discussion.
\end{example}

The reader is referred to \cite{HPS}, \cite{EqStatesCenter} for more on center isometries. For us it will be important that 
each bundle $E^*, *\in\{s,c,u,cs,cu\}$ is integrable to some ($T-$invariant) lamination $\FolW{*}$ of smooth leaves. See \cite{PartSurv}.

\begin{notation}
 If $T$ is a center isometry we write
    \[
    B_{\eps}^{*}(x)=B_{\eps}^{\FolW{*}}(x)=\{y\in \FolW{*}(x): \dis[\FolW{*}]{x}{y}<\eps\} \qquad *\in\{s,c,u,cs,cu\}.
    \]
    Above $\dis[\FolW{*}]{x}{y}$ denotes the intrinsic distance inside $\FolW{*}(x)$.
\end{notation}

Our strategy is the following. Starting from a $\pi$-quasimorphism $L:\Gamma\to\R^N$, we have shown how to construct a quasicocycle $\ms S$ for the action of $G$ on $M'$. Restricting this quasicocycle along the dynamics generated by some particular center isometry, we will show that $\ms S$ is bounded, and hence deduce that $L$ is bounded. The goal of this section is to develop the necessary tools for this argument.

The method consists of adapting ideas from thermodynamic formalism, in particular developing a Gibbs-type framework for center isometries. While similar ideas were used in \cite{TermoEqSt} in the hyperbolic setting, the present context requires a different approach, as the dynamics here is far from uniformly hyperbolic.

Some weak hyperbolicity however is given by plaque expansivity. We need the following.

\begin{proposition}[\cite{PartSurv}]
	If $T$ is a center isometry, then $(T,\FolW{c})$ is plaque expansive.
\end{proposition}

For studying certain invariant measures associated to the system  we will use the next property.

\begin{definition}\label{def:specification}
	We say that $T:X\to X$ has the specification property if for each $\delta>0$ there exists $M(\delta)$ such that the following holds: if $I_1,\dots, I_k$ are intervals with $d_{\R}(I_i,I_j)\geq M(\delta)$ for $i\neq j$ and $x_1,\dots, x_k\in X$, $I_i=[a_i,b_i]$, $b_i-a_i=n_i$, then there exists $z\in X$ so that 
	\[
	\dis{T^{j-a_i}z}{T^{j-a_i}x_i}<\delta \text{ for }j\in I_i.
	\]

    If $T$ preserves $\ms{W}$, we say that $T:X\to X$ has the quasiperiodic specification property if $z$ can be chosen so that $T^{l}z\in B_{\delta}^{\ms W}(z)$, for $l=\sum_{i=1}^k n_i+kM$. 
\end{definition}

The specification property is due to Bowen (cf. \cite{Bowen1974}), where he required furthermore that the specifying point $z$ be periodic, something natural in the context of hyperbolic dynamics. In the cases that we are interested, periodic points are not guaranteed to exist, which is the reason for introducing a natural substitute, as it will allow us to work with quasiperiodic points (defined below); it will be shown that these always exist for center isometries.

\begin{definition}\label{def:quasiperiodic}
	A finite orbit segment $\{x, Tx, \ldots, T^{n-1}x\}$ is called $(n,\delta)$-quasiperiodic if $T^n(x)\in B_{\delta}^{\ms W}(x)$. 

   We say that $x$ is $(n,\delta)$-quasiperiodic if it belongs to an $(n,\delta)$-quasiperiodic orbit segment, and that it is $\delta$-quasiperiodic if it is $(n,\delta)$-quasiperiodic, for some $n\in\N$.
\end{definition}

The following is clear.

\begin{lemma}
	Suppose that $(T,\ms{F})$ has the quasiperiodic specification property.
    \begin{enumerate}
    	\item For $\delta>0$ sufficiently small, the set of $\delta$-quasiperiodic points is dense.
	   \item $T$ is topologically mixing, meaning: for $U,V \subset X$ non-empty open sets, there exists $n_{U,V}$ such that $\forall n\geq n_{U,V}$, $T^nU\cap V\neq\emptyset.$
    \end{enumerate}
\end{lemma}

Next we give our main example of systems having the quasiperiodic specification property. Fix $T:X\toit$ a center isometry, and define the sets ($x\in X, \eps>0, d\in\N$)
\[
J^{csu}_d(\eps,x)=\bigcup_{y\in B^{cs}_\eps(x)} T^{-d}B^u_\eps(T^dy).
\]
Below we adapt the proof of specification for ergodic group automorphism of compact Abelian groups due to Lind \cite{Lind_1978} to this more general context.

\begin{lemma}
	Suppose that $X$ is compact and let $T:X\toit$ be a center isometry whose unstable lamination is minimal. Then, for every $\eps>0$ there is $M(\eps)>0$ such that for every pair of points $x_1, x_2$, and for every $n\geq M(\eps_0)$ and $d\geq 0$, there is $u_2\in B^{cs}_\eps(x_2)$
    \[
     T^{n}\left(B^u_\eps(x_1)\right) \cap J^{csu}_d(\eps,x_2)\supset T^{-d}\left(B^u_\eps(T^d u_2)\right).
    \]
\end{lemma}

\begin{proof}
	This is Lemma $6.2$ in Lind's paper. It is a direct consequence of (uniform) minimality of $\FolW{u}$ together with expansion of these leaves.
\end{proof}

\begin{proposition}\label{pro:especificacioncenterisometry}
	In the hypotheses above, $T$ has the specification property.
\end{proposition}

\begin{proof}
    Fix $\eps>0$ (small) and consider the number $M(\eps)$ given in the previous lemma. Suppose that we are given points $x_i, i=1,\ldots, k$ and corresponding intervals $I_i=[a_i,b_i], n_i=b_i-a_i$, as in the definition of specification property, with $\dis[\R]{I_i}{I_j}\geq M(\eps)$. Denote $d=\sum_{i=1}^k n_i+kM$. Let $C_i=J^{csu}_d(\eps,x_i)$. Denote $D_i=T^{-n_i}(B_u(\eps,T^{n_i}u_i))$ where the $u_i$ are going to be chosen inductively. Since $T$ acts isometrically on $\FolW{c}$ we have that 
	\begin{equation} \label{eq6.2}
		\diam(T^jC_i)\leq 2\eps\;\;(0\leq j\leq d_i, 1\leq i\leq k)
	\end{equation}

	Let $u_1=x_1$: recall that $a_2-b_1\geq M(\eps)$, hence by the previous lemma there exists $u_2\in B_{cs}(\eps,x_2)$  such that 
	\begin{align*}
	T^{a_2-a_1}(D_1) \cap C_2&=T^{a_2-b_1}\left(S^{n_1}(D_1)\right)\cap C_2=T^{a_2-b_1}\left(B^u_\eps(T^{n_2}u_1)\right)\cap C_2\\
	&\supset T^{-d_2}\left(B^u_\eps(T^{d_2}u_2)\right)=D_2.
	\end{align*}
	
	Now we repeat, and find $u_{i+1}\in  B^{cs}_{\eps}(x_{i+1})$ such that 
	\[
	T^{a_{i+1}-a_i}(D_i) \cap C_{i+1}\supset D_{i+1}.
	\]
	Hence, 
	\[
	D_1\supset T^{-a_2-a_1}(D_2)\supset\dots\supset T^{-a_k-a_1}(D_k).
	\]
	Choose $u\in T^{-a_k-a_1}(D_k)$ and let $z=T^{-a_1}u$. It follows that $T^{a_i}(z)\in C_i, \forall i$  and hence from \eqref{eq6.2}, that 
	\[
	\dis{T^n z}{T^nx_i}\leq 2\eps\;\;(a_i\leq n\leq b_i,\,1\leq i\leq k).
	\]
\end{proof}

We are ready to establish the main technical tool of this part.

\begin{theorem}\label{thm:qmspecification}
Let $T:X\toit$ be a center isometry with minimal unstable lamination. Then $T$ has the quasiperiodic specification property.
\end{theorem}

\begin{proof}
We use the same notation as in the previous proposition. It is no loss of generality to assume that $\dis{T^nz}{z}\leq \eps$, where $n=l+M$. It remains to show that $z$ can be approximated well by a $(n,\delta)-$quasiperiodic point. Denote $\lam=\max\{\norm{D^sT},\norm{D^uT}^{-1}\}$.

Consider $A=B^{s}_{\frac{\eps}{2}}(z)$: $\dis{T^nA}{z}\leq 2\eps$, thus if $h: B^{s}_{\eps}(T^nz)\to B^{s}_{2\eps}(T^nz)$ denotes the projection along the $cu$-direction, the map $h\circ T^n$ sends $A$ into itself, hence it has a fixed point $y\in A$. Note that $\dis{T^jz}{T^jy}\leq C\lam^j \frac{\eps}{2}$ for every $j\geq 0$, and $y\in B_{cu}(T^ny,3\eps)$. Since $T$ is a center isometry it follows that $T^{n-j}y\in B^{cu}_{3\eps}(T^{-j}y), \forall j=0,\ldots, n-1$.

Replace $T$ by $T^{-1}$ and apply the same argument for the point $y$: we obtain $w\in B^u_{\eps}(y)$ so that $\dis{T^{-j}y}{T^{-j}w}\leq C\lam^j\eps$ for all $j\geq 0$, and thus
\[
	\dis{T^jy}{T^jw}\leq 5\eps, j=0,\ldots, n.
\]
Moreover, $w\in B^{cs}_{4\eps}(T^{-n}w)\cap \Wcu{T^{-n}w}$. The point $w$ is the one we seek.
\end{proof}

\subsection{Pressure} 
\label{sub:pressure}

We now introduce a notion of pressure adapted to quasicocycles. In this part we assume that $T:X\toit$ is a center isometry, acting on the compact space $X$.

\begin{notation}
For $x,y\in X, n\in\N, \eps>0$ we denote
\begin{align*}
&\dis[n]{x}{y}=\max\{\dis{T^ix}{T^iy}:0\leq i\leq n-1\}\\
&\dis[\pm n]{x}{y}=\max\{\dis{T^ix}{T^iy}:0\leq |i|\leq n\}\\
&B(x,n,\eps)=\{y\in X:\dis[n]{x}{y}\leq \eps \}\\
&B^{\pm}(x,n,\eps)=\{y\in X:\dis[\pm n]{x}{y}\leq \eps\}.
\end{align*}
We say that $E \subset X$ is
\begin{itemize}
 		\item $(n,\eps)$-spanning if $X = \bigcup_{x\in E} B(x,n,\eps)$,
 		\item $(n,\eps)$-separated if $d_{n}(x,y)>\eps$ for every $x\neq y\in E$.
 	\end{itemize}
 	Note that since $X$ is compact there exists some minimal size of $(n,\eps)$-spanning sets, and some maximal size of $(n,\eps)$-separated subsets. Such quantities are related to the topological entropy of the map $T$, which we denote $\htop(T)$. See \cite{EquSta} for details and basic properties.
 \end{notation}

Given $\eps>0$ and for a fixed quasicocycle $\bm{S}=(S_n)$ we define
\begin{align}
Z_n(\eps)=\sup\left\{\sum_{x\in E}\exp(S_n(x)):E\text{ is }(n,\eps)-\text{separated}\right\}.
\end{align}

Following arguments of Bowen \cite{Bowen1974} and Franco \cite{Franco_1977}, we will analyze the behavior of $(Z_n=Z_n(\eps))_{n}$. Our quasicocyles will have the following important property.

	\begin{definition}
	We say that a sequence $\bm{S}=(S_n:X\to \R)$ has the Bowen property if there exist $\eps_0>0$ and $C>0$ such that whenever $y\in B(x,n,\eps_0)$, we have 
	\[
	|S_n(x)-S_n(y)|\leq C.
	\]
	The smallest of such constants $C$ is denoted by $\norm[B]{\bm S}$.
\end{definition}

It will be useful to consider a rough equivalence relation on quasicocycles.

\begin{definition}
	We say that a quasicocycle $(S_n)$ is trivial if there is a constant $C>0$ such that $\norml{S_n}\leq C$. We say that two quasicocycles $(S_n)$ and $(S'_n)$ are cohomologous if their difference is trivial.
\end{definition}

Notice that if a $\bm S$ is a quasicocycle with the Bowen property which is only defined on a dense subset, then we could extend it, up to cohomology, to the whole space in such a way that each $S_n$ is continuous.

\begin{convention}
	From now on $\bm S=(S_n)_n$ has the Bowen property, and each $S_n$ is continuous. 
\end{convention}

\begin{lemma}\label{lem:qmisbounded}
If $(S_n)_n$ is a quasicocycle, then there is $C'>0$ such that $\norml{S_n}\leq nC'$.
\end{lemma}

\begin{proof}
	Immediate from the quasicocycle property.
\end{proof}

\begin{lemma}\label{lem:particionsumas}
	Let $\eps>0$ and let $x_1,\dots,x_k\in X$, let $z_1,\dots, z_k$ be such that $z_j\in B(\eps,n,x_j)$ for $j=1,\dots k$ and $T^{n_{j}}(z_j)\in 
	B^c_{\gamma_0}(z_{j+1})$ for $j=1,\dots k-1$. Then
    \[
    \left|S_{n_1+\dots +n_k}(z_1)-\sum_{i=1}^kS_{n_i}(x_i)\right|\leq k3C.
    \]
\end{lemma}

\begin{proof} We proceed by induction on $k$. For $k=1$, $\left|S_{n_1}(z_1)-S_{n_1}(x_1)\right|\leq C\leq 3C$ by the Bowen property. 
	Now, 
\begin{align*}
		&\left|S_{n_1+\dots +n_{j+1}}(z_1)-\sum_{i=1}^{j+1}S_{n_i}(x_i)\right|\\
		&\leq \left|S_{n_2+\dots +n_{j+1}}(T^{n_1}(z_1))-\sum_{i=2}^{j+1}S_{n_i}(x_i)\right|+\left|S_{n_{1}}(z_1)-S_{n_{1}}(x_{1})\right|\\
		&+\left|S_{n_1+\dots +n_{j+1}}(z_1)-S_{n_2+\dots +n_{j}}(T^{n_1}(z_1))-S_{n_{1}}(z_1)\right|\\
		&\leq \left|S_{n_2+\dots +n_{j+1}}(z_2)-\sum_{i=2}^{j+1}S_{n_i}(x_i)\right|+\left|S_{n_2+\dots +n_{j+1}}(T^{n_1}(z_1))-S_{n_2+\dots +n_{j+1}}(z_2)\right|\\
		&+\left|S_{n_{1}}(z_1))-S_{n_{1}}(x_{1})\right|+\left|S_{n_1+\dots +n_{j+1}}(z_1)-S_{n_2+\dots +n_{j}}(T^{n_1}(z_1))-S_{n_{1}}(z_1)\right|\\
		&\leq 3Cj+C+C+C=3C(j+1).
\end{align*}
\end{proof}

\begin{lemma}\label{lem:BoundZn}
	If $\eps,\eps>0$ are sufficiently small, then there exists $C_{\eps,\eps'}$ so that for every $n\geq 0$,
	\[
	Z_n(\eps')\leq C_{\eps,\eps'} Z_n(\eps).
	\]
\end{lemma}

\begin{proof}
    The proof we present is a mild variation of Lemma $1$ of \cite{Bowen1974}. Take $\eps>0$ so that $2\eps\leq \Cexp$ and $0<\eps'\leq \eps$. By plaque expansiveness we get:

  \begin{claim}
  	there exists $N(\eps')$ so that if $\dis{T^kx}{T^ky}\leq 2\eps, |k|\leq N(\eps')$, then either
	\begin{itemize}
		\item $\dis{x}{y}\leq \eps'$, or
		\item $y\in B^c_{2\eps}(x)$.
	\end{itemize}
  \end{claim}

Choose $\al>0$ so that $\dis{x}{y}<\alpha\Rightarrow y\in B^{\pm}(x,N,\eps')$. Let $F$ be a $(n,\eps)$-separated set of maximal size, and let $E$ be a $(n,\eps')$-separated set. If $x\in E$ there exists $a(x)\in F$ so that $x\in B(a(x),n,\eps)$: for each $a\in F$ denote
\[
E_a=\{x\in E:a(x)=a\}.
\]
Observe that for every $x,y\in E_a$ we have that $\dis{T^ix}{T^iy}<2\eps, \forall i=0,\ldots, n-1$. By the previous Claim we have two possibilities, either
\begin{enumerate}
	\item $y\in B^c_{2\eps}(x)$, or
	\item $\dis{T^kx}{T^ky}\leq \eps'$ for all $k=N,\ldots, n-N-1$. In this case, since $\{x,y\}$ is $(n,\eps')$-separated, we get that either $\dis{x}{y}>\alpha$, or $\dis{T^nx}{T^ny}>\alpha$.  
\end{enumerate}

We denote by 
\begin{itemize}
	\item $M_0=\sup_x\{\#K: K \subset B^c_{2\eps}(x), \dis{y_i}{y_j}\geq \eps', \forall y, y'\in K\}$,
	\item $M_1=\sup\{\#G \subset X\times X: \forall (x,y)\neq (x',y')\in G, \max\{\dis{x,x'}, \dis{y}{y'}\}>\alpha\}$,
	\item $M=\max\{M_0,M_1\}$. 
\end{itemize}
Then $\# E_a\leq M$. Note also that for $x\in E_a$, $|S_n(x)-S_n(a)|\leq\norm[B]{\bm{S}}$, and thus
\[
	\sum_{x\in E}e^{S_n(x)}\leq Me^{\norm[B]{\bm{S}}}\sum_{a\in F}e^{S_n(a)}\leq Me^{\norm[B]{\bm{S}}}Z_n(\eps)\Rightarrow Z_n(\eps')\leq Me^{\norm[B]{\bm{S}}}Z_n(\eps). 
\]
\end{proof}

From the previous and \Cref{lem:particionsumas}, proceeding exactly as in Lemma $2$ of \cite{Bowen1974}, we get:

\begin{lemma}
	Given $\eps>0$ sufficiently small, there exists $E_{\eps}, D_{\eps}>0$ so that for $n_1, \ldots, n_k\geq 1$
	\[
	\prod_{i=1}^k E_{\eps}Z_{n_i}(\eps)\leq Z_{n_1+\cdots +n_k}(\eps)\leq  \prod_{i=1}^k D_{\eps}Z_{n_i}(\eps).
	\]
\end{lemma}

The previous lemma in particular implies that $\exists P(\eps)=\lim_{n\to+\infty}\frac{1}{n}\log Z_n(\eps)$, We define the pressure of $(T,\bm{S})$ by
\begin{align}
P(\bm{S})=\lim_{\eps\to 0}P(\eps).
\end{align}
If $\bm{S}=(S_n\varphi=\sum_{k=0}^{n-1}\varphi\circ T^k)_n$ for some $\varphi \in \Cr{X}$, then $P(\bm{S})$ coincides with the topological pressure of $(T,\varphi)$.

\begin{remark}
	Our standing assumption implies in particular that $T$ is $h$-expansive \cite{Bowen_1972}, and thus there exists $\eps_0>0$ so that $\forall 0<\eps\leq \eps_0$, $P(\eps)=P(\eps_0)$. This can be proven copying the arguments of Theorem $2.4$ of \cite{Bowen_1972}.
\end{remark}

We continue following Bowen: the previous lemma now gives (cf. Lemma 3 of \cite{Bowen1974}):

\begin{proposition}\label{pro:boundzn}
	For every $\eps>0$ small  it holds 
	\[
	\frac{e^{nP}}{D_{\eps}}\leq Z_n(\eps)\leq \frac{e^{nP}}{E_{\eps}},
	\]
	where $P=P(\bm{S})$, and $D_{\eps}, E_{\eps}$ are given in the previous lemma.
\end{proposition}

Now we want to obtain a similar bound replacing separated points by quasiperiodic ones. Fix $\delta>0$ small.

\begin{definition}
	A set $E$ consisting of $(n,\delta)$-quasiperiodic points verifying $x\neq y\in E\Rightarrow T^{k}x\not\in B_{3\delta}^c(y), \forall k=0,\ldots, n-1$ is called separated.
\end{definition}

From plaque expansiveness we get:

\begin{lemma}
Fix $\delta\leq \delta_0$. Then, for every $n$ there exists some $N_n(\delta)\in \N$ so that if $E$ consists of $(n,\delta)-$quasiperiodic points, and $\#E>N_n(\delta)$ then $\exists, x,y\in E$ so that $x\in B_{3\delta}^c(y)$.
\end{lemma}

We denote by $N_n(\delta)$ the cardinal of any separated set of $(n,\delta)$-quasiperiodic points with maximal number of elements.

\begin{proposition}
For any $0<\delta\leq\delta_0$ there exists some constants $C_{\delta}, D_{\delta}>0$ so that
\[
	C_{\delta}e^{nh}\leq N_n(\delta)\leq D_{\delta}e^{nh}, 
\]
where $h=\htop(T)$.	
\end{proposition}

\begin{proof}
 	If $E=\{x_i:i=1,\ldots, N_n(\delta)\}$ consists of $(n,\delta)$-quasiperiodic points, then it is $(n,3\delta_0)$-separated, hence
	\[
	N_n(\delta)\leq s_n(3\delta_0)\leq D_{3\delta_0}e^{nh}\leq D_{\delta}'e^{nh}. 
	\]
	Conversely, consider $M=M(\delta)$ the constant given by the specification property, and let $E$ be a maximally $(n-M,3\delta_0)$-separated set. For every $x\in E$ choose $p(x)$ an $(n,\delta)$-quasiperiodic point so that 
	\[
	\dis{T^ix}{T^ip(x)}\leq \delta, i=0,\ldots, n-M-1.
	\]
	The set $\tilde E=\{p(x):x\in E\}$ is separated, hence $N_{n}(\delta)\geq s_{n-M}(3\delta_0)\geq C_{\delta} e^{nh}$. 
\end{proof}

\begin{remark}
	 If $E_n$ is a maximally separated set consisting of $(n,\delta)$-quasiperiodic points, then $T(E_n)$ is as well. Since $T$ is a center isometry, and $\bm{S}$ has the Bowen property, it follows that there exists some uniform $C>0$ so that for any $(n,\delta)$-quasiperiodic point $x$, and $k=0,\ldots, n$
	 \[
	 |S_n(x)-S_n(T^kx)|\leq C.
	 \]
\end{remark}

Fix some family $\ms E=\{E_n:n\in\N\}$, where $E_n$ is a maximally separated set consisting of $(n,\delta)$-quasiperiodic points: we call such family $\delta$-quasiperiodic family. For $\ms E$, define
\[
	Z_n(\ms E)=\sum_{x\in E_n} e^{S_n(x)}.
\]

The same proof as above shows the following.

\begin{corollary}
	For any $\delta\leq\delta_0$ there exists some constants $C_{\delta}, D_{\delta}>0$ so that
	\[
	C_{\delta}e^{nP}\leq Z_n(\ms E)\leq D_{\delta}e^{nP}, 
    \]
    In particular, if $\ms E,\ms E'$ are $\delta$-quasiperiodic families, then $\left(\frac{Z_n(\ms E)}{Z_n(\ms E')}\right)_n$ is uniformly bounded, from above and below. 
\end{corollary}


For $x\in X$ denote $\lam_x=\Leb(\cdot|B_{6\delta}^c(x))$ (to simplify we omit the dependence with respect to $\delta$). Consider the probability measures
\[
	\mu^{\ms E}_n=\frac{1}{Z_n(\ms E)}\sum_{x\in E_n} e^{S_n}\lam_x
\]
Fix also any weak accumulation point $\mu^{\ms E}$ of $\{\mu_n^{\ms E}\}$.

\begin{lemma}
 	There exists some constant $C_{\eps,\delta}>0$ which is independent of the family $\ms E$ and of $\mu^{\ms E}$, and so that
 	\[
 	(\forall x\in X, 0<\eps<\delta),\ \mu^{\ms E}\left(B(x,n,\eps)\right)\geq  C_{\eps,\delta} e^{S_n(x)-nP}
 	\]
 \end{lemma}

\begin{proof}
Fix $R_m$ a maximally $(m,3\eps)$-separated set. By the quasiperiodic specification property, for every $y\in R_m$ there exists some $(n+m+2M,\eps)$-quasiperiodic point $p(y)\in B(x,n,\eps)$ so that $T^{n+M}p(y)\in B(y,m,\eps)$: note that $y\mapsto p(y)$ is injective. By maximality of $E_{n+m+2M}$, there exists $0\leq k(y)\leq  n+m+2M-1$ and $q(y)\in E_{n+m+2M}$ so that $p(y)\in B^c_{3\del}(T^{k(y)}q(y))$: it is no loss of generality to assume that 
$p(y)\in B^c_{3\del}(q(y))$ and therefore
\begin{align*}
&q(y)\in B(x,n,3\delta+\eps)\\
&T^{n+M}q(y)\in B(y,m,3\delta+\eps).
\end{align*}
If $q(y)=q(y')$, then necessarily $p(y)\in B^c_{6\del}(p(y'))$. We can thus consider $R_m' \subset R_m$ so that $y\in R_m'\mapsto q(y)$ is injective of maximal size, and deduce that for some constant $c$ independent of $n,\delta$, $\#R_m'\geq c \#R_m$. From the Bowen property of $\bm{S}$ it follows that
\[
	\sum_{q(y):y\in R_m'}e^{S_m(T^{n+M}q(y))}\geq ce^{-C}\sum_{y\in R_m}e^{S_m(y)}\geq C_{\eps} Z_m(\eps). 
\]

This in turn says that  
\begin{align*}
\MoveEqLeft Z_{n+m+2M}(\ms E)\cdot\mu^{\ms E}_{n+m+2M}(B(x,n,\eps))\\
&\geq \frac{\min\{\Leb(B_{c}(z,\eps)):z\in X\}}{\max\{\Leb(B_{c}(z,\eps)):z\in X\}}e^{-C}\sum_{q(y):y\in R_m'} e^{S_m(T^{n+m}q(y))}\\
&\geq C_{\eps,\delta}' Z_m(\ms E)  
\end{align*}
\[
\Rightarrow \mu^{\ms E}_{n+m+2M}(B(x,n,\eps))\geq C_{\eps,\delta} e^{S_n(x)-nP}.	
\]

From here, using that $B(x,n,\eps)$ is closed, it follows the claim.
\end{proof}

\begin{lemma}
	For $\delta>0$ sufficiently small, there exists $K_{\delta}>0$ so that 
	\[
 	(\forall x\in X, 0<\eps<\delta),\ \mu^{\ms E}\left(B(x,n,\eps)\right)\leq  K_{\delta} e^{S_n(x)-nP}
 	\]
\end{lemma}

\begin{proof}
Take $M=M(\delta)$ the specification constant. Consider $0<\eps<\frac{\delta}{6}$ so that $\dis{x}{y}<6\eps$ implies $\dis{T^ix}{T^iy}\leq \delta$, for all $|i|\leq M$. Denote $U=\bigcup_{y\in B(x,3\eps,n)} B^c_{6\del}(y)$: of course, if $y\in E_{n+m+2M}\setminus U$ then $B^c_{6\del}(y)\cap B(x,3\eps,n)=\emptyset$.

On the other hand, if $y,z\in U\cap E_{n+m+2M}$, we let $y'=B^c_{6\del}(z)\cap B(x,3\eps,n)$, $z'=B^c_{6\del}(z)\cap B(x,3\eps,n)$: since $d_n(y',z')\leq 6\eps$, it follows that
 \[
 	\dis{T^jy}{T^jz}\leq 13\delta, j=-M,\ldots, n+M. 
 \]
 The constant $\delta$ is chosen so that $13\delta<\Cexp$, so if $z\not\in B^c_{\Cexp}(y)$, necessarily $\dis[m]{y}{z}\geq 13\delta$. It follows that 
\[
	 \#\left(U\cap E_{n+m+2M}\right)\leq D_{\delta}s_m(13\delta)\leq K_{\delta}e^{mh}
\]
for some constant $K_{\delta}$ which is independent of $m\in\N$ and of the family $\ms E$. Using the Bowen property, we  can thus deduce that
 \[
  Z_{n+m+2M}(\ms E)\mu_{n+m+2M}^{\ms E}(B(x,3\eps,n))\leq K_{\delta}' e^{S_n(x)}e^{C+mP} 
 \]
 and
 \[
   \mu_{n+m+2M}^{\ms E}(B(x,3\eps,n))\leq \tilde{K}_{\delta}  e^{S_n(x)-nP}.
 \]
where $\tilde{K}_{\del}$ only depends on $\delta$. This in turn implies that
  \[
    \mu^{\ms E}(B(x,\eps,n))\leq \tilde{K}_{\delta}  e^{S_n(x)-nP}.
   \]
\end{proof}

Below the constants $K_{\delta}, C_{\delta,\eps}$ are the ones given in the previous two lemmas.

\begin{corollary}
	There exists some $T-$invariant measure $\nu$ so that for every $0<\eps<\delta, n\geq 0$,
	\[
	C_{\delta,\eps}\leq \frac{\nu(B(x,\eps,n))}{e^{S_n(x)-nP}}\leq K_{\delta}.
	\] 
\end{corollary}

\begin{proof}
	Fix $\ms E$ a $\delta$-quasiperiodic family. Since $T$ is a center isometry, note  that $T\ms E$ is also a $\delta$-quasiperiodic family. Moreover, for every $n, j$ 
	\[
	T_{*}^j\mu^{\ms E}_n=\frac{1}{Z_n(\ms E)}\sum_{x\in E_n} T_{*}^j\lam_x=\mu^{T^j\ms E}_n.
    \]
   Consider 
   \[
   \nu_n=\frac{1}{n}\sum_{j=0}^{n-1} T^{j}_*\mu^{\ms E}_n=\frac{1}{n}\sum_{j=0}^{n-1} \mu^{T^j\ms E}_n.
   \]
   Since $C_{\delta,\eps}, K_{\delta}$ do not depend on the $\delta$-quasiperiodic family, it follows that for every $n,x,\eps$,
   \[
   	C_{\delta,\eps}\leq \frac{\nu_n(B(x,\eps,n))}{e^{S_n(x)-nP}}\leq K_{\delta}.
	\]
	The same thus holds for any accumulation point $\nu\in\{\nu_n\}_n'$, which is easily seen to be $T-$invariant. 
\end{proof}

\begin{corollary}\label{cor:fullergodic}
	The measure $\nu$ of the previous theorem is of full support and ergodic.
\end{corollary}

\begin{proof}
	Any open set contains a $(n,\eps)-$Bowen ball, therefore it has positive $\nu$-measure. Ergodicity is established exactly as in Lemma $6$ of \cite{Bowen1974}, with straightforward modifications (as in \Cref{lem:BoundZn}).
\end{proof}

The above implies that $\{\nu_n\}_n$ has unique accumulation point. Not only that, if we now consider $\delta\leq\delta'$ and $\delta,\delta'-$periodic families $\ms E,\ms E'$, the two resulting measures are absolutely continuous with respect to each other, hence equal. In the end, we aim to prove that $\nu$ is the Haar measure.

\begin{definition}
	The measure $\nu$ constructed above will be referred to as the equilibrium measure of $(T,\mathbf S)$.
\end{definition}

To justify this nomenclature, we note that by \cite{Cao_2008},
\[
	P=\sup\{h_\mu(T)+\mu(\bm S):T_{*}\mu=\mu\}
\]
where
\[
	\mu(\bm S)=\lim_n \frac{1}n \int S_n \dd\mu. 
\]
At this moment we do not know that $\nu$ is the unique equilibrium measure. On the other hand, by construction, the disintegration of $\nu$ along the center foliation is equivalent to Lebesgue. It can be shown that among invariant measures whose center disintegration is absolutely continuous, the measure $\nu$ is the unique of such measures attaining the supremum in the formula above. See \cite{EqStatesCenter}.

\begin{remark}
	The above cited article gives a different proof of ergodicity of $\nu$ (in fact, Bernoulliness).
\end{remark}

Since $S_n(x) \approx nP-\log(\nu(B(x,\eps,n)))=n(h_\nu(T)+\nu(\bm S))-\log(\nu(B(x,\eps,n)))$ we deduce.

 \begin{corollary}\label{cor:sameEq}
 	Let $\bm S, \bm{S}'$ be Bowen quasicocycles with the same equilibrium measure $\nu$. Suppose that $\nu(\bm S)=\nu(\bm{S}')$. Then $\bm S, \bm{S}'$ are cohomologous.
 \end{corollary}

On the other hand, it is clear by the construction that cohomologous quasicocyles have the same equilibrium measure (perhaps with different pressure).

We end this section by showing that if $R:X\toit$ is a homeomorphism that commutes with $T$ and acts isometrically on $\FolW{c}$, then it also preserves $\nu$. 

Observe that if $\ms E$ is a $\delta$-quasiperiodic family for $T$, then $R\ms E=\{RE_n\}_n$ is also a $\delta$-quasiperiodic family. It follows that
\begin{align*}
R_*\nu=R_*\lim_n \sum_{j=0}^{n-1}T_*^{j}\mu_n^{\ms E}=\lim_n \sum_{j=0}^{n-1}(T^j)_*\mu_n^{R\ms E}=\nu 
\end{align*}
and $\nu$ is $R-$invariant.

%% file: 4ProofthmNEW.tex

\section{Proof of the main theorem}\label{sec:proofthm}

We return to the setting introduced in the first two sections. Recall that $M=\Gamma\backslash G, M'=\Gamma'\backslash G'$ and suppose that we are given a $\pi$-quasimorphism $L:\Gamma\to \R$. We aim to show that $L$ is bounded.

Our first task is proving the following Proposition.

\begin{proposition}\label{pro:existecentroisometry} 
   There exists $a_1,\cdots, a_r\in G, r=\rank_R G$ so that for every $i$, the right translation $T_{a_i}:M'\toit$ is a center isometry with minimal unstable foliation, and furthermore, there exists $S_i\approx \SL_2(\R)<G$ so that for $b\in S_i$, $T_b$ commutes with $T_{a_i}$.
\end{proposition}

The construction follows from standard arguments, but since our hypotheses on $\Gamma$ are very general (and in particular, we are not assuming irreducibility), we give the complete proof.

Recall that we are assuming, after reductions, that $G$ is algebraically simply connected. Write
\[
	G=\prod_{j=1}^pG ^{(j)}\quad \Gamma=\prod_{j=1}^p\Gamma^{(j)}
\]
where each $\Gamma^{(j)}<G^{(j)}$ is irreducible, and $\rank_{\R}G^{(j)}\geq 2$.

For each  irreducible block $G^{(i)}$, choose a Cartan involution and let
\[
\lie g^{(i)} = \lie k^{(i)}\oplus\lie p^{(i)}
\]
be the corresponding Cartan decomposition of its Lie algebra. Let $K^{(i)}<G^{(i)}$ be the maximal compact subgroup with Lie algebra
$\lie k^{(i)}$, and choose a maximal abelian subspace
\[
\lie a^{(i)}\subset\lie p^{(i)}.
\]
Define
\[
A^{(i)} \defeq \exp\bigl(\lie a^{(i)}\bigr).
\]

Let $\Sigma_i\subset(\lie a^{(i)})^*$ be the corresponding restricted root system. Choose a set of positive restricted roots
$\Sigma_i^+\subset\Sigma_i$, and denote by $\Delta_i$ the associated set of simple restricted roots. If
\[
\lie g^{(i)}_\beta =
\left\{
X\in\lie g^{(i)}:[H,X]=\beta(H)X\text{ for every }H\in\lie a^{(i)}
\right\}
\]
is the restricted root space corresponding to $\beta$, define
\[
\lie n^{(i)}=\bigoplus_{\beta\in\Sigma_i^+} \lie g^{(i)}_\beta,\qquad
N^{(i)}=\exp\bigl(\lie n^{(i)}\bigr).
\]
By the Iwasawa decomposition, 
\[
G^{(i)}=K^{(i)}A^{(i)}N^{(i)}.
\]

Take products,
\[
K=\prod_{i=1}^pK^{(i)},
\qquad
\mathfrak a=\bigoplus_{i=1}^p\mathfrak a^{(i)},
\qquad
A=\prod_{i=1}^pA^{(i)},
\qquad
N=\prod_{i=1}^pN^{(i)}.
\]
We then have the Iwasawa decomposition
\[
G=KAN.
\]

For each $i$, let
\[
\lie a^{(i),+}=\left\{
H\in\lie a^{(i)}:\alpha(H)>0\text{ for every }\alpha\in\Delta_i
\right\}
\]
be the open positive Weyl chamber, and let
$\overline{\lie a^{(i),+}}$ denote its closure. Define
\[
\lie a^+ = \bigoplus_{i=1}^p\lie a^{(i),+},
\qquad
\overline{\lie a^+}=\bigoplus_{i=1}^p\overline{\lie a^{(i),+}}
\qquad
A^+=\exp\bigl(\overline{\lie a^+}\bigr).
\]
We then have the Cartan decomposition 
\[
G=KA^+K;
\]
we consider the Cartan projection
\[
    \kappa_A: G\rightarrow A^+
\]
associated with the choices of $K$, $A$, and the positive Weyl chamber. Under the product decomposition of $G$, it is given
componentwise by
\[
\kappa_A(g_1,\ldots,g_p)=\bigl(\kappa_{A^{(1)}}(g_1),\ldots,\kappa_{A^{(p)}}(g_p)\bigr).
\]

Let
\[
    \mc{W}_i = \{H_i\in \overline{\lie a^{(i),+}}\setminus\{0\}:\al(H_i)=0\text{ for some simple restricted root }\al\}=\partial \overline{\lie a^{(i),+}}\setminus\{0\}.
\]
The real-rank assumption $\dim\lie{a}_i\geq 2$ implies that $\mc W_i$ spans $\lie{a_i}$. Indeed, the fundamental coweights form a basis of $\lie a^{(i)}$ and lie in the closed positive Weyl chamber. Each is annihilated by every simple root except the corresponding one to it, so the rank assumption ensures that each fundamental coweight belongs to $\mc W_i$.

Let $\mc W \defeq \mc W_1\times\cdots\times\mc W_p\subset\mathfrak a$. By construction, every $H=(H_1,\ldots,H_p)\in\mc W$ satisfies $H_i\neq0$ for every $i$, and for each $i$ there exists a simple restricted root $\alpha_i\in\Delta_i$ such that $\alpha_i(H_i)=0$.

We claim that $\Span_{\R}\mc W=\lie a$. To see this, fix some index $i$, and for every $j\neq i$ choose $w_j\in \mc W_j$. Given $v\in \mc W_i$ both
\begin{align*}
H=(w_1,\ldots, w_{i-1},v,w_{i+1},\ldots, w_p)\quad H'=(w_1,\ldots, w_{i-1},2v,w_{i+1},\ldots, w_p)
\end{align*}
are in $\mc W$, therefore $H'-H\in \Span_{\R}\mc W$. Thus $\lie a_i=\Span_{\R}\mc W_i \subset \Span_{\R}\mc W$, which implies our claim. In particular,
there exists 
\[
 H^{(1)},\ldots, H^{(r)}\in \mc W    
\]
basis of $\lie a$. Defining $a_j=\exp(H^{(j)})$ we get
\begin{align}\label{eq:generadorAgood}
A=\langle\exp(\R H^{(1)}),\ldots, \exp(\R H^{(r)})\rangle.
\end{align}

For what follows we fix $a=a_j=\exp(H)$. Consider the right translation
\[
T_a:M=\Gamma\backslash G\longrightarrow\Gamma\backslash G,
\qquad
T_a(x)=xa.
\]
Denote
\begin{align*}
&\lie g^s = \bigoplus_{\beta(H)>0}\mathfrak g_\beta\\
&\lie g^{c} = \lie g_0\oplus \bigoplus_{\beta(H)>0}\mathfrak g_\beta\qquad \lie g_0=Z_{\lie g}(\lie a)\\
&\lie g^u = \bigoplus_{\beta(H)<0}\mathfrak g_\beta.
\end{align*}

\begin{lemma}
   $T_a$ is a center isometry. 
\end{lemma}

\begin{proof}
     Choose an inner product on $\lie g$ making the restricted-root decomposition $\lie g=\lie g^s\oplus \lie g^c\oplus \lie g^u$ to be orthogonal, and descend the corresponding left-invariant metric to $M$. In the left trivialization of the tangent bundle,
     \[
     DT_a=\Ad(a^{-1}),
     \]
     and therefore, on $\lie g_{\beta}$ we have $DT_a=e^{-\beta(H)}\Id$. This gives stable, center, and unstable bundles for $T_a$. 
\end{proof}
In fact, every non-identity element $a'\in A$ induces a center isometry. See \cite{Hof1985}.

\begin{remark}
Note that for our choice of $a$ (which is not regular), the center leaves strictly contain the corresponding $A$-orbit. Indeed, $H$ is singular, therefore there exists a restricted root $\alpha$ with $\alpha(H)=0$. On $\lie g_{\alpha}$, $\Ad(a^{-1})$ acts as the identity, providing center directions beyond $\lie a$.
\end{remark}

The stable horospherical subgroup of $a$ is
\[
N_a^+ \defeq \left\{u\in G: a^{-n}ua^n\longrightarrow e\text{ as }n\longrightarrow+\infty \right\}.
\]
Note that for $u\in N_a^+$, 
\[
T_a^n(xu)=xua^n=xa^n\bigl(a^{-n}ua^n\bigr),
\]
and hence the points $x$ and $xu$ approach each other under forward iteration, and $xN_a^+ \subset \Ws{x;T_a}$. Since $N_a^+$-orbits are tangent to the stable bundle of $T_a$, and hence form its stable foliation. Thus $xN_a^+ = \Ws{x;T_a}$. Note also that $N_a^+$ is contained in the unipotent subgroup
\[
N=\exp\left(\bigoplus_{\beta\in\Sigma^+}\mathfrak g_\beta\right)
\]
appearing in the Iwasawa decomposition $G=KAN$.

For every $i$, let
\[
N_{a,i}^+=N_a^+\cap G^{(i)}.
\]
Since $H_i\neq0$ and the restricted roots span $(\lie a^{(i)})^*$, there exists $\beta\in\Sigma_i$ such that $\beta(H_i)>0$. Thus $N_{a,i}^+$ contains a nontrivial one-parameter unipotent subgroup and is therefore non-compact. It is also closed. By Moore's ergodicity theorem (Chapter $2$ of \cite{semisimpleZimmer}) and the irreducibility of $\Gamma^{(i)}$, the group $N_{a,i}^+$ acts ergodically on $\Gamma^{(i)}\backslash G^{(i)}$ with respect to the normalized Haar measure. Since
\[
N_a^+=\prod_{i=1}^pN_{a,i}^+,
\]
the product action of $N_a^+$ on $M=\Gamma\backslash G$ is ergodic as well.

Next we work with the compact bundle extension $M'$ of $M$, which has fiber $\mc C=\overline{\Im\pi}$. Recall that $G'=G\times\mc C$. The maximal split subgroup of $G'$ relevant to us is
\[
A'=A\times\{e\},
\]
and, with respect to the maximal compact subgroup $K\times \mc C$, its Cartan projection is
\[
\kappa_{A'}(g,k)=\kappa_A(g).
\]
Then $a\in A \subset A'$ also defines a center isometry $T_a:M'\toit$ that lifts the previous one through $p_M:M'\to M$.

The action of $N_a^+$, through its inclusion in the first factor of $G'$, is also ergodic on $M'$. To see this, let
\[
N_i\triangleleft G^{(i)}
\]
be the connected normal subgroup generated by $N_{a,i}^+$. Since $N_{a,i}^+\subseteq N_i$, the previously established ergodicity of $N_{a,i}^+$ implies that the $N_i$-action on $\Gamma^{(i)}\backslash G^{(i)}$ is ergodic. As the normalized Haar measure has full support, this action has a dense orbit. Normality of $N_i$ then gives
\[
\overline{\Gamma^{(i)}N_i}=G^{(i)}.
\]
Writing 
\[
N=\prod_{i=1}^pN_i,
\]
we obtain
\[
\overline{\Gamma N}=G.
\]
\begin{claim}
\[
S\defeq\overline{\Gamma'(N\times\{e\})}=G\times\mc C.
\]
\end{claim}

We have that $N\times\{e\}\triangleleft G\times \mc C$, therefore $S$ is a closed subgroup. Let $q_G:S\to G, q_{\mc C}: S\to \mc C$ be the corresponding projections onto the first and second factor.

Compactness of $\mc C$ implies that $q_G$ is a closed proper map. Since $q_G(S) \supset \Gamma N$, which is dense in $G$, we get that $q_G(S)=G$.

To deal with the projection onto the second factor (where $q_{\mc C}$ may not be closed), we consider 
\[
D=\{d\in\mc C:(e,d)\in S\}.
\]
Then $D$ is closed. If $c\in q_{\mc C}(S)$, $c=q_{\mc C}(g,c)$ for some $g\in G$, and
\[
    d\in D\Rightarrow (g,c)(e,d)(g,c)^{-1}=(e,cdc^{-1})\in S \therefore cdc^{-1}\in D.
\]
Hence, $q_{\mc C}(S)$ normalizes $D$. Since this projection contains $\pi(\Gamma)$, which is dense in $\mc C$, we get that $D$ is normal in $\mc C$.

Consider the group homomorphism  
\[
\varphi \colon S\longrightarrow\mc C/D,
\qquad
\varphi(g)=cD\quad\text{whenever }(g,c)\in S.
\]
Then $\varphi\circ q_{G}$ is continuous, and since $q_G$ is proper, we deduce that $\varphi$ is continuous. Since $G$ is connected and semisimple with no compact factors, every continuous homomorphism from $G$ to a compact group is trivial. Thus $S=G\times D$.

Since $\Gamma'\subseteq S$, we have $\pi(\Gamma)\subseteq D$, and density gives $D=\mc C$. This proves the claim.

\smallskip 
 
Now we conclude using the Mautner phenomenon:

\begin{theorem}[Moore, \cite{Moore_1980}]
    Let $G$ be a connected semisimple real Lie group with finite center, and let $U$ be a connected subgroup generated by one-parameter unipotent subgroups. Let $N$ be the smallest closed normal subgroup of $G$ containing $U$.

    Then, for every strongly continuous unitary representation $\rho:G\to\Uop(\Hil)$, a vector $f\in \Hil$ is fixed by $U$ if and only if it is fixed by $N$.
\end{theorem}

By the above, every $N_a^+$-invariant vector $f\in \mc{L}^2(M')$ is invariant under $N\times\{e\}$. Lift $f$ to a locally square-integrable function $\widetilde f$ on $G'$. This lift is left $\Gamma'$-invariant and right $(N\times\{e\})$-invariant. Since $N\times\{e\}$ is normal in $G'$, it is also left $(N\times\{e\})$-invariant, and hence left invariant under $\Gamma'(N\times\{e\})$. By density and continuity of translations in $\mc L^2_{\mathrm{loc}}(G')$, $\widetilde f$ is invariant under all left translations of $G'$. Thus $f$ is constant almost everywhere, proving that the $N_a^+$-action on $M'$ is ergodic.

\smallskip 
 
We use the obtained ergodicity to deduce minimality of the corresponding stable foliation. An ergodic horospherical action on a compact homogeneous space is minimal; see \cite{Ellis_1978}. Therefore,
\[
\overline{xN_a^+}=\overline{\Ws{x;T_a}}=M
\qquad
\text{and}
\qquad
\overline{x'N_a^+}=\overline{\Ws{x;T_a}}=M'
\]
for every $x\in M$ and every $x'\in M'$, where on $M'$ we identify $N_a^+$ with $N_a^+\times\{e\}$. Thus, both right translations $T_a:M\toit, T_a:M'\toit$ have minimal stable foliations.

Finally, we construct a subgroup isomorphic to $\SL_2(\R)$ in the centralizer of $a$. For each $i$, choose $\alpha_i\in\Delta_i$ with $\alpha_i(H_i)=0$, and set
\[
\beta_i=
\begin{cases}
\alpha_i,&\text{if }2\alpha_i\notin\Sigma_i,\\
2\alpha_i,&\text{if }2\alpha_i\in\Sigma_i.
\end{cases}
\]
Then $\beta_i$ is nonmultipliable ($2\beta_i\not\in\Sigma_i$) and $\beta_i(H_i)=0$.

The standard restricted-root construction gives an $\mathfrak{sl}_2$-triple
\[
X_i\in\mathfrak g^{(i)}_{\beta_i},
\qquad
Y_i\in\mathfrak g^{(i)}_{-\beta_i},
\qquad
[X_i,Y_i]=H_{\beta_i}\in\mathfrak a^{(i)},
\]
where $H_{\beta_i}$ is the coroot vector. This triple integrates to a homomorphism
\[
\iota_i:\SL_2(\R)\longrightarrow G^{(i)}
\]
whose restriction to the diagonal torus is the restricted coroot $\beta_i^\vee$:
\[
\iota_i\!\begin{pmatrix}t&0\\0&t^{-1}\end{pmatrix}
=\beta_i^\vee(t)
\qquad(t\in\R^\times).
\]
Since $G^{(i)}$ is algebraically simply connected and $\beta_i$ is nonmultipliable, $\beta_i^\vee$ is primitive in the cocharacter lattice of the underlying split algebraic torus \cite[Theorem~S.3.1]{ConradFeng2020}. Thus there exists a character $\chi_i$ of the split algebraic torus satisfying
$\langle\chi_i,\beta_i^\vee\rangle=1$, and 
\[
\chi_i\bigl(\iota_i(-I)\bigr)
=\chi_i\bigl(\beta_i^\vee(-1)\bigr)
=-1,
\]
so $\iota_i(-I)\neq e$. Since the differential of $\iota_i$ is injective, its kernel is discrete and therefore central in $\SL_2(\R)$. Hence $\ker\iota_i\subseteq\{\pm I\}$, and $\iota_i$ is injective.

Moreover, $\beta_i(H_i)=0$ implies that $X_i$ and $Y_i$ commute with $H_i$, and $H_{\beta_i}$ does as well. Therefore
\[
\iota_i\bigl(\SL_2(\R)\bigr)
\subseteq C_{G^{(i)}}(\exp H_i)^\circ.
\]

Define
\[
\iota:\SL_2(\R)\longrightarrow G,
\qquad
\iota(s)=\bigl(\iota_1(s),\ldots,\iota_p(s)\bigr).
\]
Then $\iota$ is injective, and its image $S$ satisfies
\[
S\cong\SL_2(\R),
\qquad
S\subseteq C_G(a)^\circ.
\]
Its projection to every irreducible block is noncompact. Replacing $a$ by $a^{-1}$ (to obtain minimal unstable foliations) we finish the proof of \Cref{pro:existecentroisometry}.

\paragraph{The equilibrium state of $T_a$.} We fix $\ell:\R^N\to\R$ a linear functional. Let also $S$ be the quasicocycle for the action $G\acts \Sigma'\times G'$ given on \Cref{pro:quasicoccycleG}, and denote by $(S^{a,\ell}_n(\cdot))=(\ell(S(\cdot,a^n)))$ the (real) quasicocycle over the action of $T_a$.

\begin{lemma}
$(S^{a,\ell}_n)_n$ has the Bowen property.  
\end{lemma}

\begin{proof}
Since $M'$ is a manifold, the fundamental domain $\Sigma'$ can be chosen to be a polyhedron. If $\tilde p,\tilde q$ lie in the same Bowen ball in $M'$,
then we can find lifts $p,q\in G'$ that lie in nearby translates of $\Sigma'$, say $p\in \gamma_p\Sigma'$ and $q\in \gamma_q\Sigma'$ with $\gamma_p\Sigma'\cap\gamma_q\Sigma'\neq\emptyset$. Likewise, the end-points of the their $a^n$-orbit lie in nearby fundamental domains. It follows that 
\[
    a^{n}p=b a^n{q}c
\]
where $b,c$ belong to a finite set of $\Gamma$. The quasimorphism property of $L$ then yields the result.  
\end{proof}

Let $\nu$ be the associated equilibrium state. In particular, for every $b\in G$ commuting with $a$, $(T_b)_*\nu=\nu$. Define 
\[
 V \defeq \overline{\inner{A}{S}}^\circ < C_G(a)^{\circ}   
\]
Then $V$ is generated by split semisimple elements and $\Ad$-unipotent elements, therefore it is of $\Ad$-triangular type. Note that $V\acts (M',\nu)$ ergodically. 

Let $V_u$ be the smallest closed normal subgroup of $V$ containing all its $\Ad$-unipotent one-parameter subgroups. Since $S$ is generated by upper- and lower-triangular subgroups, $S \subset V_u$. By construction, the projection of $S$ to each irreducible block $G^{(i)}$ is the noncompact subgroup
$\iota_i(\SL(2,\mathbb R))$. Consequently, the projection of $V_u$ to at least one almost-simple factor of every irreducible block is
noncompact. 

\begin{remark}
    If we knew that $V_u\acts (M',\nu)$ ergodically, then we would be able to use Ratner's theorem \cite{Ratner_1991} and deduce that $\nu$ is the Haar. As ergodicity at this stage is not clear, we need some extra steps.  
\end{remark}

We use the following theorem due to Margulis--Tomanov \cite{Margulis_1996} and Mozes--Weiss \cite{Mozes_1998}, in the form stated as Theorem $3.1$ by Starkov \cite{Starkov_1999}: there exists a connected closed subgroup $U<G'$ so that $V_u<U$, for which almost every $V_u$-ergodic component of $\nu$ is the $U$-invariant probability measure on a closed $U$-orbit in $M'$. Moreover, the group
\[
    P=\{p\in G':g\mapsto pgp^{-1}\text{ preserves }U\text{ and its Haar measure}\}
\] 
contains $V$ (in particular, $S$), and there is a point $x\in M'$ so that its $P$-orbit is closed, and $\nu(xP)=1$. Since $\nu$ has full support, we deduce that $M'=xP$, and with no loss of generality we can assume that $P$ is connected. In particular $G=G\times\{e\} \subset P$, and $G$ normalizes $U$. 

Let $J=\mrm{pr}_G(U)$: then $J\triangleleft G$ is normal and closed, since $\mrm{pr}_G$ is closed. As 
\[
    S \subset V_u \subset U
\]
and $S$ is connected, its projection to $G$ is contained in $J^{o}$. Hence $J^{o}\triangleleft G$ is a connected normal subgroup whose projection to every irreducible block $G^{(i)}$ is noncompact. In particular, 
\[
    J_i=J^{o}\cap G^{(i)}
\]
is a nontrivial normal subproduct of $G^{(i)}$, containing at least one of its almost-simple. Since $\Gamma^{(i)}<G^{(i)}$ is irreducible
\[
    \overline{\Gamma^{(i)}J_i}=G^{(i)}.
\]
It follows that every $J^0$-orbit, and hence every $J$-orbit, in $M=\Gamma\setminus G$ is dense. On the other hand, if $xU \subset M'$ is one of the closed $U$-orbits, its projection to $M$ is a $J$-orbits. Since the fibers of $M'\to M$ are compact, the projection of the closed set $xU$ is closed. Thus, this $J$-orbit is both closed and dense, hence equal to $M$. Therefore, $\Gamma J=G$.

We claim that $J=G$. Otherwise, consider $q:G\to G/J$ the quotient map, and note that $q(\Gamma)=G/J$. The right side is connected and positive-dimensional, while $q(\Gamma)$ is countable, which gives a contradiction. 

\begin{claim}
$G\times\{e\}\subseteq U$. 
\end{claim}

Indeed, let $\lie u=\operatorname{Lie}(U)$. Then, since $G\times\{e\}$ normalizes $U$ and $\mrm{pr}_G(U)=G$,
\begin{align*}
[\lie g\oplus\{0\},\lie u]\subseteq\lie u\\
\operatorname{pr}_{\lie g}(\lie u)=\lie g,
\end{align*}
while semisimplicity of $\lie g$ give
\[
    \lie g\oplus \{0\}=[\lie g, \lie g]\oplus\{0\}.
\]
Given $X,Y\in \lie g$ choose $Z\in\operatorname{Lie}(\mc C)$ so that $(Y,Z)\in\lie u$, and compute
\[
    ([X,Y],0)=[(X,0),(Y,Z)]\in \lie u.
\]
Connectedness of $G$ now gives $G\times\{e\}\subseteq U$.

Recall that $\Gamma'$ is the graph of $\pi$ and that $\overline{\pi(\Gamma)}=\mathcal C$. Therefore,
\[
\Gamma'(G\times\{e\})=G\times\pi(\Gamma),
\]
and hence
\[
G'=\overline{\Gamma'(G\times\{e\})}\subseteq\overline{\Gamma'U}.
\]
The $U$-orbit through $\Gamma'e$ is closed, so its preimage $\Gamma'U$ in $G'$ is closed. Thus $\Gamma'U=G'$, proving that $U$ acts transitively on $M'$.

Almost every $V_u$-ergodic component of $\nu$ is $U$-invariant, so $\nu$ itself is $U$-invariant. By uniqueness of the invariant probability measure on a transitive homogeneous space, we deduce that $\nu$ is the normalized Haar measure on $M'$.

\paragraph{From the equilibrium state to boundedness}Let
\[
I(a)=\lim_{n\to+\infty}\frac{1}{n}\int S(x,a^n)m(\dd x),
\]
where $m$ denotes Haar measure on $M'$. One checks directly that $I:A\to\R^N$ is a homomorphism. In particular,
\[
    I(a^t)=tI(a)
\]
along every one-parameter subgroup $a^t=\exp(t\log a)$. Fix $a=a_j$ one of the good singular elements as in \eqref{eq:generadorAgood}, and $\ell\in (\R^N)^*$. Consider the constant quasicocycle $J=(J_n)$ with $J_n(x)=\ell(I(a^n))=n\ell(I(a))$: its unique equilibrium state is $m$, and 
\[
    m(S^{a,\ell})=\lim_{n\to\oo} \frac{1}{n} \int \ell(S(x,a^n)) m(\dd x)=\ell(I(a))=m(J)  
\]
hence by \Cref{cor:sameEq} we deduce that
\[
    \sup_{n\geq 0}\norml{\ell\left(S(\cdot,a^n)-I(a^n)\right)}<\oo.
\]
Choosing a basis $\ell_1,\ldots, \ell_N$ of $(\R^N)^*$, we get
\[
	\sup_{n\geq 0}\norml{S(\cdot,a^n)-I(a^n)}<\oo.
\] 
Applying the quasicocycle property to the relation $a^na^{-n}=e$, we then deduce that
\[
    \sup_{n\in\Z}\norml{S(\cdot,a^n)-I(a^n)}<\oo.
\] 
Write $t=n+s$ where $n\in\Z, 0\leq s<1$. Then 
\begin{align*}
\abs{S(x,a^t)-I(a^t)}\leq \abs{S(x,a^n)-I(a^n)}+C_S+2\sup_{x\in\Sigma', 0\leq s<1}\abs{S(x,a^s)}.
\end{align*}

\begin{lemma}
    It holds $\sup_{x\in\Sigma', 0\leq s<1}\abs{S(x,a^s)}<\oo$.
\end{lemma}

\begin{proof}
   Consider the compact set $K_a=\{a^s:0\leq s\leq 1\}$ and recall that for $x\in \Sigma', g\in G$ 
   \[
   \gamma=[xa^s]\Rightarrow xa^s\in \gamma\Sigma'
   \] 
   and therefore $\gamma \Sigma'\cap \Sigma'K_a\neq \emptyset$. It follows that
   \[
   \{[xk]:x\in\Sigma',k\in K_a\} \subset \{\gamma\in\Gamma:\gamma \Sigma'\cap \Sigma'K_a\neq \emptyset\},
   \]
   the later being a finite set, since $\overline{\Sigma'}$ is compact. Since $\abs{S(x,g)}=\abs{L([xg])}$, the claimed boundedness follows.
\end{proof}

Considering $a_1,\ldots, a_r$ that generate $A$ as in \eqref{eq:generadorAgood}, and using again the quasicocycle property, we deduce that
\[
	C_1:=\sup_{a\in A}\norml{S(\cdot,a)-I(a)}<\oo.
\]

We may choose the fundamental domain $\Sigma$ to be right $K$-invariant, i.e. $\Sigma K=\Sigma$. Indeed, let $q:G\to G/K$ be the quotient map and choose a fundamental domain $\overline{\Sigma} \subset G/K$, for the $\Gamma$-action, which with no loss of generality can be taken as a polyhedron. Let $\Sigma=q^{-1}(\overline{K})$: then $\Sigma K=\Sigma$. Note that $\mrm{Stab}_{\Gamma}(gK)=\Gamma\cap gKg^{-1}$ is trivial, since $\Gamma$ is without torsion and $K$ is compact. It follows that if we define 
\[
    [g]=\gamma\in\Gamma\text{ so that }\gamma^{-1}gK\in\overline\Sigma
\]
there is only ambiguity on the chosen elements for points $g$ landing in the finitely many boundary faces of $\overline K$. Selecting arbitrarily some of its faces, the return element is unique and the return map is well defined. It thus follows that $[gk]=[g]$ for every $g\in G, k\in K$, and in particular $[xk]=e$ for every $k\in K$ and $(x,z)\in\Sigma'$. The formula inducing the quasicocycle then gives
\[
S((x,z),(k,c))=z^{-1}L([xk])=z^{-1}L(e)
\]
hence
\[
   |S((x,z),(k,c))|=|L(e)|. 
\]

Next we identify $\Sigma'$ with $\Gamma'\backslash G'$ and use the right $G$ action on $\Gamma'\backslash G'$. Using the $G=KA^+K$ decomposition and the quasicocycle property, we see that there is a constant $C_0$ ($=2\|\delta_\pi(L)\|$) so that for $g=k_1ak$, $k_i\in K$, $a\in A^+$, $p=\Gamma'(x,z)\in\Gamma'\backslash G'$, 
\begin{align*}
\MoveEqLeft\Big|S\left(q,(k_1ak_2,e)\right)-S\left(pk_1a,z),(k_2,e)\right)-S(p(k_1,e),(a,e))-S\left(p,(k_1,e)\right)\Big|\leq C_0.
\end{align*}
Hence we get that 
\[
\left|S(p,(k_1ak_2,e))-S(p(k_1,e),(a,e))\right|\leq C_0+2|L(e)|	
\]
which in turn implies
\[
	\left|S(p,(g,e))-I(\kappa_A(g))\right|\leq C_0+2|L(e)|+C_1
\]
for every $g\in G, p=(x,z), x\in \Sigma, z\in \mc C$ (recall that $\kappa_A(g)\in A^+$ is the Cartan projection).

Using the explicit expression for $S$, we have
\[
	\left|z^{-1}L([xg])-I(\kappa_A(g))\right|\leq  C_0+2|L(e)|+C_1:=C_2
\]
for every $g\in G, x\in \Sigma$, $z\in \mc C$, and in particular,
\begin{align*}
\left|\left(z^{-1}-\Id\right)L([xg])\right|\leq 2C_2.
\end{align*}

 Fix $x_0\in \Sigma$. Given $\gamma\in \Gamma$, take $g=x_0^{-1}\gamma x_0$, and thus $[x_0g]=\gamma$. It follows that 
 \begin{align}\label{eq:zquasimorphism}
 \sup_{\gamma\in\Gamma, z\in\mc C}\left|\left(z^{-1}-\Id\right)L(\gamma)\right|\leq 2C_2.
\end{align}

This permits us to make some reduction. At the beginning we split $\R^N=E\oplus F$ with $E=\Fix{\mc C}, F=E^{\perp}$. Then $\pi$ acts trivially on $E$, whereas $\bigcap_{z\in\mc C}\ker (z-\Id)|F=\{0\}$.  Thus, by \eqref{eq:zquasimorphism}, on the $F$ part we have that $L$ is uniformly bounded, and we can focus on the $E$ part. It is no loss of generality then to assume then that $\pi$ is trivial, and by choosing a basis of $E^*$ and considering each resulting coordinate, we can also assume that $N=1$. Then we have
\begin{align*}
 \left|L([xg])-I(\kappa_A(g))\right|\leq C_2
\end{align*}
for every $g\in G$ and $x\in\Sigma$, which is equivalent to 
\[
    |S(p,a)-I(\kappa_A(g))|\leq C_2
\]
for every $a\in A$. Therefore
\[
    |I(a)-I(\kappa_A(g))|\leq C_1+C_2\leq 2C_2.
\]
Applying the same to $a^n$, since $\kappa_A(a^n)=n\kappa_A(a)$, we deduce that $I(a)=I(\kappa_A(a))$.

Recall that the Weyl Group is $W=N_K(A)/Z_K(A)$ where $N_K(A), Z_K(A)$ are the normalizer and centralizer of $A$ in $K$, respectively. $W$ acts on $A$ by conjugation, and for $a\in A$ we have $W\cdot a\cap A^+=\{\kappa_A(a)\}$. Thus $\kappa_A(w(a))=\kappa_A(a)$, and therefore
\[
    I(a)=I(\kappa_A(a))=I(\kappa_A(w(a)))=I(w(a)).
\]
We deduce that $aw(a)^{-1}\in\ker I$, for every $a\in A, w\in W$.

It is simpler to lift everything to $\lie a$ using the exponential map. Define
\[
\widetilde I:\lie a\longrightarrow\R,
\qquad
\widetilde I(X)=I(\exp X).
\]
The map $\widetilde I$ is additive and $W$-invariant. Recall that $W$ is generated by the reflections $s_\alpha$ across the
hyperplanes $\ker\alpha$, where $\alpha$ ranges over the simple restricted roots. Hence
\[
\widetilde I\bigl(X-s_\alpha(X)\bigr)=\widetilde I(X)-\widetilde I(s_\alpha(X))=0.
\]
Since
\[
\Im(\Id-s_\alpha)=(\ker\alpha)^\perp,
\]
the map $\widetilde I$ vanishes on $(\ker\alpha)^\perp$ for every
simple restricted root $\alpha$. Moreover,
\[
\sum_{\alpha\in\Delta}(\ker\alpha)^\perp
=\left(\bigcap_{\alpha\in\Delta}\ker\alpha\right)^\perp
=\mathfrak a,
\]
because the simple restricted roots span $\mathfrak a^*$. It follows
that $\widetilde I\equiv0$, and therefore $I\equiv0$.

We deduce that,
\[
|L([xg])|\leq C_2
\]
for every $x\in\Sigma$ and $g\in G$. Since every $\gamma\in\Gamma$ can be written as $[xg]$ for suitable $x\in\Sigma$ and $g\in G$, we
conclude that
\[
|L(\gamma)|\leq C_2
\qquad
(\gamma\in\Gamma).
\]
Thus $L$ is bounded.

\section*{Acknowledgments}

We thank Francesco Fournier-Facio for several helpful suggestions and comments, particularly concerning the references. OpenAI's Codex,
using the GPT-5.6 Sol model, was used during the preparation of this manuscript to improve the exposition and assist in checking the
arguments. The authors take full responsibility for all mathematical content.

P.D. Carrasco was partially supported by CNPq (Produtividade em Pesquisa) 305197/2025-8, Projeto Universal CNPq 401737/2025-0 and FAPEMIG APQ-02338-24. 

F. Rodriguez-Hertz was partially supported by NSF grant DMS-2453688.